\documentclass[11pt,a4paper]{article}

\usepackage{a4wide}
\usepackage{amssymb,amsmath,amsthm}
\usepackage[english]{babel}
\usepackage{t1enc}
\usepackage[latin2]{inputenc}
\usepackage{epsfig}
\usepackage{subfigure}
\usepackage{graphicx}

\usepackage{hyperref}

\newtheorem{thm}{Theorem}
\newtheorem{corollary}[thm]{Corollary}
\newtheorem{lemma}[thm]{Lemma}

\newtheorem{claim}[thm]{Claim}

\newtheorem{problem}[thm]{Problem}
\newtheorem{obs}[thm]{Observation}
\newtheorem{defi}[thm]{Definition}

\usepackage{indentfirst}
\usepackage{xspace} 

\def\U{\mbox{\ensuremath{\mathcal U}}\xspace}
\def\C{\mbox{\ensuremath{\mathcal C}}\xspace}
\def\F{\mbox{\ensuremath{\mathcal F}}\xspace}
\def\HH{\mbox{\ensuremath{\mathcal H}}\xspace} 
\def\I{\mbox{\ensuremath{\mathcal I}}\xspace}

\def\D{\mbox{\ensuremath{\mathcal D}}\xspace}
\def\B{\mbox{\ensuremath{\mathcal B}}\xspace}
\def\F{\mbox{\ensuremath{\mathcal F}}\xspace}

\begin{document}
	
	\title{Coloring intersection hypergraphs of pseudo-disks}
	\author{Bal\'azs Keszegh\thanks{Research supported by the by the Lend\"ulet program of the Hungarian Academy of Sciences (MTA), under the grant LP2017-19/2017 and by the National Research, Development and Innovation Office -- NKFIH under the grant K 116769.}
	}
	
	
	\maketitle
	
\begin{abstract}
	We prove that the intersection hypergraph of a family of $n$ pseudo-disks with respect to another family of pseudo-disks admits a proper coloring with $4$ colors and a conflict-free coloring with $O(\log n)$ colors. Along the way we prove that the respective Delaunay-graph is planar.
	 We also prove that the intersection hypergraph of a family of $n$ regions with linear union complexity with respect to a family of pseudo-disks admits a proper coloring with constantly many colors and a conflict-free coloring with $O(\log n)$ colors. Our results serve as a common generalization and strengthening of many earlier results, including ones about proper and conflict-free coloring points with respect to pseudo-disks, coloring regions of linear union complexity with respect to points and coloring disks with respect to disks.
\end{abstract}

\section{Introduction}

Proper colorings of hypergraphs and graphs defined by geometric regions are studied extensively due to their connections to, among others, conflict-free colorings and to cover-decomposability problems, both of which have real-life motivations in cellular networks, scheduling, etc. For applications and history we refer to the surveys \cite{surveycd, surveycf}. Here we restrict our attention to the (already long list of) results that directly precede our results. We discuss further directions and connections to other problems in Section \ref{sec:discussion}.

Our central family of regions is the family of pseudo-disks. Families of pseudo-disks have been regarded in many settings for a long time due to being the natural way of generalizing disks while retaining many of their topological and combinatorial properties.  Problems regarded range from classic algorithmic questions like finding maximum size independent (disjoint) subfamilies \cite{pseudodisjoint} to classical combinatorial geometric questions like the Erd\H os-Szekeres problem \cite{pseudoerdosszekeres}.
Probably the most important pseudo-disk family is the family of homothets of a convex region. Pseudo-disks are also central in the area of geometric hypergraph coloring problems as we will see soon in this section.

\smallskip
Before we state our results and their precursors, we need to introduce some basic definitions.

\begin{defi}
	Given a hypergraph $\HH$, a \textbf{proper coloring} of its vertices is a coloring in which every hyperedge $H$ of size at least $2$ of $\HH$ contains two differently colored vertices.
\end{defi}

Throughout the paper hypergraphs can contain hyperedges of size at least $2$ only, in other words we do not allow (or automatically delete, when necessary) hyperedges of size $1$.


\begin{defi}
	Given a hypergraph $\HH$, a \textbf{conflict-free coloring} of its vertices is a coloring in which every hyperedge $H$ contains a vertex whose color differs from the color of every other vertex of $H$.
\end{defi}

\begin{defi}
	Given a family $\B$ of regions and a family $\F$ of regions, the \textbf{intersection hypergraph} of $\B$ \textbf{with respect to} (wrt. in short) $\F$ is the simple (that is, it has no multiple edges) hypergraph $\I(\B,\F)$ which has a vertex $v_B$ corresponding to every $B\in \B$ and for every $F\in \F$ it has a hyperedge $H=\{v_B:B\cap F\ne \emptyset\}$ (if this set has size at least two) which corresponds to $F$. The \textbf{Delaunay-graph} of $\B$ wrt. $\F$ is the graph on the same vertex set containing only the hyperedges of $\I(\B,\F)$ of size $2$.
\end{defi}

Note that if $\B$ is finite, then even if $\F$ is infinite, $\I(\B,\F)$ has finitely many hyperedges. In particular, a hyperedge $H$ can correspond to multiple members of $\F$.

We also define the following subgraph of the Delaunay-graph:

\begin{defi}
	The \textbf{restricted Delaunay-graph} of $\B$ wrt. $\F$ is the subgraph of the Delaunay-graph containing only those (hyper)edges $H_F=\{v_{B_1},v_{B_2}\}$ for which the corresponding $F \in \F$ intersects $B_1$ and $B_2$ in disjoint regions, that is, $F\cap B_1\ne \emptyset$, $F\cap B_2\ne \emptyset$, $F\cap B_1\cap B_2=\emptyset$ and $F\cap B=\emptyset$ for every $B\in \B\setminus\{ B_1,B_2\}$.
\end{defi}

\begin{obs}[\protect\cite{smorod-onthe}]\label{obs:delaunay}
	If $\B$ and $\F$ are families for which every hyperedge of $\I(\B,\F)$ contains a hyperedge of size $2$, then a proper coloring of the Delaunay-graph is also a proper coloring of  $\I(\B,\F)$.
\end{obs}

In the literature hypergraphs that have the property assumed in Observation \ref{obs:delaunay} are sometimes called \textit{rank two} hypergraphs (e.g., in \cite{smorod-onthe}).

\begin{defi}
	We say that $\B$ can be properly/conflict-free/etc. colored wrt.  $\F$ with $f(n)$ colors if for any $\B'$ subfamily of $\B$ of size $n$, the hypergraph $\I(\B' ,\F)$ can be properly/conflict-free/etc. colored with $f(n)$ colors.
\end{defi}

We note that one could further generalize this notion such that instead of intersection we regard inclusion/reverse-inclusion as the relation that defines the hyperedges, as it is done, e.g., in \cite{oktans9} for hypergraphs defined by intervals. For problems related to such variants see Section \ref{sec:discussion}.

In the majority of earlier results, coloring points wrt. a family of regions or coloring a family of regions wrt. points were regarded. While they were not defined as intersection hypergraphs, they do fit into our general definition of intersection hypergraphs, choosing either $\B$ or $\F$ to be the family of points of the plane. Notice that whenever every point is or can be added as a member of the family $\F$, e.g., for disks, for homothets of a convex region and for pseudo-disks\footnote{See Corollary \ref{cor:shrinking-closed} for a precise formulation of this property for pseudo-disks.}, coloring intersection hypergraphs of $\F$ wrt. to $\F$ is a common generalization both of coloring points wrt. $\F$ and coloring $\F$ wrt. points.

There are only a few earlier results regarding this general setting. In \cite{oktans9} intersection (and also inclusion and reverse-inclusion) hypergraphs of intervals of the line were considered. In \cite{smorod-int}  and \cite{fekete-int} they considered intersection hypergraphs (and graphs) of (unit) disks, pseudo-disks, squares and axis-parallel rectangles.

\subsection{Results related to pseudo-disks}

It is well known that the Delaunay-graph of points wrt. disks is planar and thus proper $4$-colorable, and Observation \ref{obs:delaunay} implies that the respective hypergraph is also proper $4$-colorable.

In \cite{smorod-onthe} Smorodinsky developed a general framework (based on the framework presented in \cite{evenlotker}): using proper colorings of the subhypergraphs of the hypergraph with constantly many colors it can build a conflict-free coloring with $O(\log n)$ colors (where $n$ is the number of vertices of the hypergraph). The results mentioned from now on use this framework to get a conflict-free coloring once there is a proper coloring.
First, using the precursor of this framework, it was proved by Even et al. \cite{evenlotker} that points wrt. disks admit a conflict-free coloring with $O(\log n)$ colors:

\begin{thm}[\protect\cite{evenlotker}]\label{thm:ptwrtdisk}
	Let $\D$ be the family of disks in the plane and $P$ a finite set of points, $\I(P,\D)$ admits a conflict-free coloring with $O(\log n)$ colors, where $n=|P|$.
\end{thm}

\begin{defi}
	A \textbf{Jordan region} is a (simply connected) closed bounded region whose boundary is a closed simple Jordan curve. 	
	
	A family of Jordan regions is called a family of \textbf{pseudo-disks} if the boundaries of every pair of the regions intersect in at most two points.
\end{defi}

Although we could not find it in the literature, it is well known that if instead of a family of disks we take a family of pseudo-disks, the bound of Theorem \ref{thm:ptwrtdisk} still holds:

\begin{thm}[folklore]\label{thm:ptwrtpsdisk}
	Let $\F$ be a family of pseudo-disks and $P$ a finite set of points, $\I(P,\F)$ admits a proper coloring with $4$ colors and a conflict-free coloring with $O(\log n)$ colors, where $n=|P|$.
\end{thm}

The proof is that the Delaunay-graph of points with respect to pseudo-disks is also a planar graph (implied by, e.g., Lemma \ref{lem:romdisjoint}, see later) and thus proper $4$-colorable and then we can apply Observation \ref{obs:delaunay} (we can assume that the hypergraph is rank two by, e.g., Corollary \ref{cor:shrinking-closed}) to conclude that the respective hypergraph is also $4$ colorable. Then the general framework of Smorodinsky mentioned above implies an $O(\log n)$ upper bound for a conflict-free coloring.


The dual of Theorem \ref{thm:ptwrtdisk} was also proved by Even et al. and was generalized to pseudo-disks by Smorodinsky:

\begin{thm}[\protect\cite{evenlotker}]\label{thm:diskwrtpt}
	Let $P$ be the set of all points of the plane and $\B$ a finite family of disks, $\I(\B,P)$ admits a proper coloring with $4$ colors and a conflict-free coloring with $O(\log n)$ colors, where $n=|\B|$.
\end{thm}

\begin{thm}[\protect\cite{smorod-onthe}]\label{thm:psdiskwrtpt}
	Let $P$ be the set of all points of the plane and $\B$ be a finite family of pseudo-disks, $\I(\B,P)$ admits a proper coloring with a constant number of colors and a conflict-free coloring with $O(\log n)$ colors, where $n=|\B|$.	
\end{thm}

While for coloring pseudo-disks wrt. points there was no explicit upper bound, for the special case of homothets of a convex region\footnote{For further results about homothets of a convex region see Section \ref{sec:discussion}.}, the constant was shown by Cardinal and Korman to be $4$, just like for disks:

\begin{thm}[\protect\cite{cardinalkorman}]\label{thm:homotwrtpt}
	Let $P$ be the set of all points of the plane and $\B$ be a finite family of homothets of a given convex region, $\I(\B,P)$ admits a proper coloring with $4$ colors.	
\end{thm}

Recently it was proved by Keller and Smorodinsky that disks w.r.t. disks can also be colored in such a way, a common generalization of Theorems \ref{thm:ptwrtdisk} and \ref{thm:diskwrtpt}:

\begin{thm}[\protect\cite{smorod-int}]\label{thm:diskwrtdisk}
	Let $\D$ be the family of all disks in the plane and $\B$ a finite family of disks, $\I(\B,\D)$ admits a proper coloring with $6$ colors and a conflict-free coloring with $O(\log n)$ colors, where $n=|\B|$.
\end{thm}

While they did not prove the same for pseudo-disks, they solved two special cases (they stated these only as part of the proof of their main result):

\begin{claim}[\protect\cite{smorod-int}]\label{claim:psdisk-speccases}
	Given a finite family $\F$ of pseudo-disks and a subfamily $\B$ of $\F$, If either $\B$ or $\F\setminus \B$ contains only pairwise disjoint pseudo-disks, then $\I(\B,\F)$	admits a proper coloring with a constant number of colors and a conflict-free coloring with $O(\log n)$ colors, where $n=|\B|$.
\end{claim}

In this paper we generalize Theorem \ref{thm:diskwrtdisk} to the case of coloring a pseudo-disk family wrt. another pseudo-disk family, which is a common generalization of all the above results (Theorems	 \ref{thm:ptwrtdisk},\ref{thm:diskwrtpt},\ref{thm:psdiskwrtpt},\ref{thm:homotwrtpt},\ref{thm:diskwrtdisk} and Claim \ref{claim:psdisk-speccases}). Moreover we prove the optimal upper bound of $4$ colors, which improves the bound of Theorem \ref{thm:psdiskwrtpt} for coloring pseudo-disks wrt. disks from some constant number of colors to $4$ colors and improves the bound of Theorem \ref{thm:diskwrtdisk} for coloring disks wrt. disks from $6$ colors to $4$ colors. Furthermore, it provides an alternative proof for Theorems \ref{thm:diskwrtpt} and \ref{thm:homotwrtpt} (both of these were originally proved using dualization and solving equivalent problems about coloring points in the $3$ dimensional space):

\begin{thm}\label{thm:psdiskwrtpsdisk}
	Given a family $\F$ of pseudo-disks and a finite family $\B$ of pseudo-disks, $\I(\B,\F)$
	admits a proper coloring with $4$ colors.
\end{thm}

Along the way we prove that the respective Delaunay-graph is planar:

\begin{thm}\label{thm:psdiskdelplanar}
	Given a family $\F$ of pseudo-disks and a finite family $\B$ of pseudo-disks, the Delaunay-graph of $\B$ wrt. $\F$ is a planar graph.
\end{thm}

This was not known even for the Delaunay-graph of pseudo-disks wrt. points.
With standard methods (present also in the proof of Theorem \ref{thm:linunionwrtpsdisk}) this already implies Theorem \ref{thm:psdiskwrtpsdisk} with $6$ colors instead of $4$. To achieve the optimal bound we will need some additional ideas.

We mention that if $\B$ and $\F$ are both families of pairwise disjoint simply connected regions, then $\I(\B,\F)$ is a planar hypergraph\footnote{The most common definition of a planar hypergraph is that its bipartite incidence graph is planar.} and all planar hypergraphs can be generated this way. From this perspective Theorem \ref{thm:psdiskwrtpsdisk} says that even with a much more relaxed definition of planarity of a hypergraph it remains $4$-colorable.

Using the usual framework it easily follows from Theorem \ref{thm:psdiskwrtpsdisk} that:

\begin{corollary}\label{cor:psdiskcf}
	Given a family $\F$ of pseudo-disks and a finite family $\B$ of pseudo-disks, $\I(\B,\F)$
	admits a conflict-free coloring with $O(\log n)$ colors, where $n=|\B|$.
\end{corollary}

We note that Claim \ref{claim:psdisk-speccases} implies in a straightforward way the main result of $\cite{smorod-int}$ about conflict-free coloring the (open/closed) neighborhood hypergraphs of intersection graphs of pseudo-disks (for definitions and details see $\cite{smorod-int}$). Thus, Theorem \ref{thm:psdiskwrtpsdisk} also implies their main result.

Note that in Theorem \ref{thm:psdiskwrtpsdisk} $\B$ and $\F$ are not related in any way, thus among others, even though two convex regions can intersect infinitely many times, it implies somewhat surprisingly the following:

\begin{corollary}
	We can proper color with $4$ colors the family of homothets of a convex region $A$ wrt. the family of homothets of another convex region $B$.
\end{corollary}

Buzaglo et al. \cite{buzaglo} have shown that the VC-dimension of the hypergraph $\I(P,\F)$, defined by a finite point set $P$ with respect to a family $\F$ of pseudo-disks, is at most $3$ (and this bound is tight) and used this to prove that the number of hyperedges of size at most $k$ in such a hypergraph is $O(k^2n)$ (and this can be attained already by $\F$ being a specific family of disks). More recently, Aronov et al. \cite{aronov18} proved independently from us with an almost identical proof the case of Lemma \ref{lem:psdiskownpointdelplanar} when $\B=\F$ (for details see later, it is a very special case of Theorem \ref{thm:psdiskdelplanar})  and showed that this implies that for a pseudo-disk family $\F$ and a finite subfamily $\B$ of $\F$, the intersection hypergraph $\I(\B,\F)$ has VC-dimension at most $4$ (and this bound is tight). Using methods similar to the ones in \cite{buzaglo} they show that this implies that the number of hyperedges of size at most $k$ in $\I(\B,\F)$ in this case is $O(k^3n)$ (they do not give a construction with a matching lower bound, however). We  show that using Lemma \ref{lem:psdiskownpointdelplanar} instead of their weaker variant we get the following more general statements (where $\B$ is not necessarily a subfamily of $\F$):

\begin{thm}\label{thm:vcdim}
	Given a family $\F$ of pseudo-disks and a finite family $\B$ of pseudo-disks, the intersection hypergraph $\I(\B\,\F)$ has VC-dimension at most $4$ (and this bound is tight).
\end{thm}

\begin{thm}\label{thm:ksets}
	Given a family $\F$ of pseudo-disks and a finite family $\B$ of pseudo-disks, the number of hyperedges of size at most $k$ in the intersection hypergraph $\I(\B\,\F)$ is $O(k^3n)$.
\end{thm}

We note that in case of Theorem \ref{thm:ksets} we are again not aware of a matching lower bound (and in fact this upper bound might easily not be tight).

\subsection{Results related to families of linear union complexity}

In \cite{smorod-onthe} Theorem \ref{thm:psdiskwrtpt} was shown by proving a more general statement about coloring a family of regions that has linear union complexity wrt. points.

\begin{defi}
	
	Let $\B$ be a family of finitely many Jordan regions in the plane such that the boundaries of its members intersect in finite many points. The \textbf{vertices} of the \textbf{arrangement} of $\B$ are the intersection points of the boundaries of regions in $\B$, the \textbf{edges} are the maximal connected parts of the boundaries of regions in $\B$ that do not contain a vertex and
	the \textbf{faces} are the maximal connected parts of the plane which are disjoint from the edges and the vertices of the arrangement.
\end{defi}

\begin{defi}
	The \textbf{union complexity} $\U (\B)$ of a family of Jordan regions $\B$ is the number of edges of the arrangement $\B$ that lie on the boundary of $\cup_{B\in \B} B$.
	
	We say that a family of regions $\B$ has \textbf{($c$)-linear union complexity} if there exists a constant $c$ such that for any subfamily $\B'$ of $\B$ the union complexity of $\B'$ is at most $c|\B'|$.\footnote{Sometimes in the literature, in the definition of union complexity they count vertices rather than edges, yet it is easy to see (see Lemma \ref{lem:unioncompldef} later) that this does not affect the property of having linear union complexity when $\B$ is a family of Jordan regions. Also note that linear union complexity is not always defined hereditarily, we defined it this way in order to simplify our statements.}
\end{defi}

\begin{thm}[\protect\cite{Kedem1986}]\label{thm:psislinunion}
	Any finite family of pseudo-disks in the plane has a linear union
	complexity.
\end{thm}

Theorem \ref{thm:psislinunion} shows that the following result of Smorodinsky is indeed more general than Theorem \ref{thm:psdiskwrtpt}.
\begin{thm}[\protect\cite{smorod-onthe}]\label{thm:linunionwrtpt}
	Let $P$ be the set of all points of the plane and $\B$ be a finite family of Jordan regions with linear union complexity. Then $\I(\B,P)$ admits a proper coloring with a constant number of colors and a conflict-free coloring with $O(\log n)$ colors, where $n=|\B|$.\footnote{When a statement is about a family with ($c$-)linear union complexity, by a constant we mean a constant that depends on $c$ and the $O$ notation similarly hides a dependence on $c$.}
\end{thm}

We generalize this in the following way:

\begin{thm}\label{thm:linunionwrtpsdisk}
	Given a family $\F$ of pseudo-disks and a finite family $\B$ of Jordan regions with linear union complexity, $\I(\B,\F)$ admits a proper coloring with a constant number of colors.
\end{thm}

Note that using Theorem \ref{thm:psislinunion} we get that Theorem \ref{thm:linunionwrtpsdisk} implies Theorem \ref{thm:psdiskwrtpsdisk} with a (non-explicit and worse) upper bound.

\begin{corollary}\label{cor:linunioncf}
	Given a family $\F$ of pseudo-disks and a finite family $\B$ of Jordan regions with linear union complexity, $\I(\B,\F)$
	admits a conflict-free coloring with $O(\log n)$ colors, where $n=|\B|$.
\end{corollary}

We note that the proof of Theorem \ref{thm:linunionwrtpsdisk}  works without modification also when instead of Jordan-regions (i.e., bounded regions whose boundary is a single closed Jordan curve), $\B$ is a family of bounded regions whose boundary is a union of a finite number of disjoint Jordan-curves (in particular the regions are not necessarily simple or connected). Observe that the definition of union complexity is well-defined for such families as well, yet in this case we do need to count edges and not vertices to have a true statement (recall that previously it did not matter). Thus Theorem \ref{thm:linunionwrtpsdisk} and  Corollary \ref{cor:linunioncf} hold in this more general setting as well.

The rest of the paper is structured as follows.
In Section \ref{sec:properties} we prove Theorem \ref{thm:psdiskdelplanar}. In Section \ref{sec:psdisk} we prove Theorem \ref{thm:psdiskwrtpsdisk}. 
Implications about conflict-free colorings and VC-dimensions and the number of bounded size hyperedges in intersection hypergraphs is detailed in Section \ref{sec:cfvc}.
In Section \ref{sec:linunion} we prove the results about regions of linear union complexity. We also give an example that shows that we cannot always color properly points wrt. a family of linear union complexity with constantly many colors. This shows that Theorem \ref{thm:linunionwrtpsdisk} is strongest possible in the sense that we cannot change $\F$ from pseudo-disks to a family with linear union complexity, even if $\B$ would be only a finite family of points in the plane. Finally, in Section \ref{sec:discussion} we discuss some related (open) problems.

\section{The Delaunay-graph of pseudo-disks wrt. pseudo-disks}\label{sec:properties}

In this section we prove  Theorem \ref{thm:psdiskdelplanar}. First we list some tools we need from the papers of Pinchasi \cite{pinchasi} and Buzaglo et al. \cite{buzaglo} about pseudo-disks:

\begin{lemma}[\protect\cite{pinchasi}]\label{lem:shrink}
	Let $\F$ be a family of pseudo-disks. Let $D\in  \F$ and let $x\in D$ be any point. Then $D$ can continuously be shrunk to the point $x$ so that at each moment $\F$ is a family of pseudo-disks.
\end{lemma}

Note that when shrinking this way, we can keep all shrunk copies of $D$ in the family, it remains a pseudo-disk family as their boundaries are pairwise disjoint.

\begin{defi}
	We say that a pseudo-disk family $\F$ is \textbf{shrinking-closed} (for some family of regions $\B$) if we cannot add a new pseudo-disk $F$ to $\F$ which is contained in a pseudo-disk $F'$ already in $\F$ (i.e., $F\subset F'$) in a way that strictly increases the number of hyperedges in $\I(\B,\F)$.
\end{defi}

 We can enlarge $\F$ greedily until it becomes shrinking-closed (we need to add at most $2^{|\B|}$ pseudo-disks):

\begin{corollary}\label{cor:shrinking-closed}
	Given a family $\F$ of pseudo-disks and a finite family $\B$ of regions, there exists a shrinking-closed (for $\B$) family of pseudo-disks $\F'$ with $\F '\supseteq \F$ and consequently with $\I(\B,{\F'})\supseteq \I(\B,\F)$. In particular, for every point $p$ which is in some $F\in \F$, $H_p=\{v_B:p\in B\in \B\}$ is a hyperedge of $\I(\B,\F')$ (which we call the hyperedge that corresponds to $p$) if it is of size at least two.
\end{corollary}

The second part of Corollary \ref{cor:shrinking-closed} holds as for a point $p$ contained in some $F\in \F$ we can apply Lemma \ref{lem:shrink}. 


By Corollary \ref{cor:shrinking-closed} when proving our results, it will be enough to consider the case when $\F$ is shrinking-closed as adding hyperedges to a hypergraph cannot remove edges from its Delaunay-graph and cannot decrease the number of colors needed for a proper (or conflict-free) coloring. The following corollary of Lemma \ref{lem:shrink} will be useful when we deal with a shrinking-closed pseudo-disk family:

\begin{corollary}[\protect\cite{pinchasi}]\label{cor:shrink2}
	Let $\B$ be a family of pairwise disjoint regions in the plane and let $\F$ be a family of
	pseudo-disks. Let $D$ be a member of $\F$ and suppose that $D$ intersects exactly $k$ members of $\B$ one
	of which is the set $B\in \B$. Then for every $2\le l\le k$ there exists a set $D'\subset D$ such that $D'$
	intersects $B$ and exactly $l-1$ other regions from $\B$, and $\F\cup\{D'\}$ is again a family of pseudo-disks.
\end{corollary}

\begin{lemma}[\protect\cite{buzaglo}]\label{lem:disjointedges}
	Let $D_1$ and $D_2$ be two pseudo-disks in the plane. Let $x$ and $y$ be two
	points in $D_1\setminus D_2$. Let $a$ and $b$ be two points in $D_2\setminus D_1$. Let $e$ be any Jordan
	arc connecting $x$ and $y$ that is fully contained in $D_1$. Let $f$ be any Jordan arc connecting $a$ and $b$
	that is fully contained in $D_2$. Then $e$ and $f$ cross an even number of times.
\end{lemma}

\begin{lemma}[\protect\cite{pinchasi}]\label{lem:romdisjoint}
	Given a family $\F$ of pseudo-disks and a finite family $\B$ of pairwise disjoint connected sets in the plane, the Delaunay-graph of $\B$ wrt. $\F$ is a planar graph.
\end{lemma}

Note that for a family $\B$ of pairwise disjoint connected sets, the Delaunay-graph and the restricted Delaunay-graph are the same. Assuming also that $\B$ is a family of Jordan regions, Lemma \ref{lem:romdisjoint} becomes a special case of Theorem \ref{thm:psdiskdelplanar}. Before proving Theorem \ref{thm:psdiskdelplanar} we prove another special case which for Jordan regions strengthens Lemma \ref{lem:romdisjoint}, while its proof remains relatively simple:

\begin{lemma}\label{lem:psdiskownpointdelplanar}
	Given a family $\F$ of pseudo-disks and a finite family $\B$ of pseudo-disks such that each member $B\in B$ contains a point which is in no other $C\in \B$, the Delaunay-graph of $\B$ wrt. $\F$ is a planar graph.
\end{lemma}

\begin{proof}
	
	We will draw $G$ in the plane in such a way that every pair of edges in $G$
	that do not share a common vertex cross an even number of times. The Hanani-Tutte Theorem \cite{hanani1,hanani2} then implies the planarity of $G$. In our case the edges may have self-crossings too but the Hanani-Tutte theorem holds in this case as well as we can easily redraw the edges to avoid self-crossings (for further details see, e.g., \cite{hananirado}).
	
	For every $v_B, B\in \B$ choose a point $p_B\in B$ which is in no other $C\in \B$. For an illustration of the rest of the proof see Figure \ref{fig:psdiskownpoint}.

	\begin{figure}
		\centering
		\includegraphics[height=5.9cm]{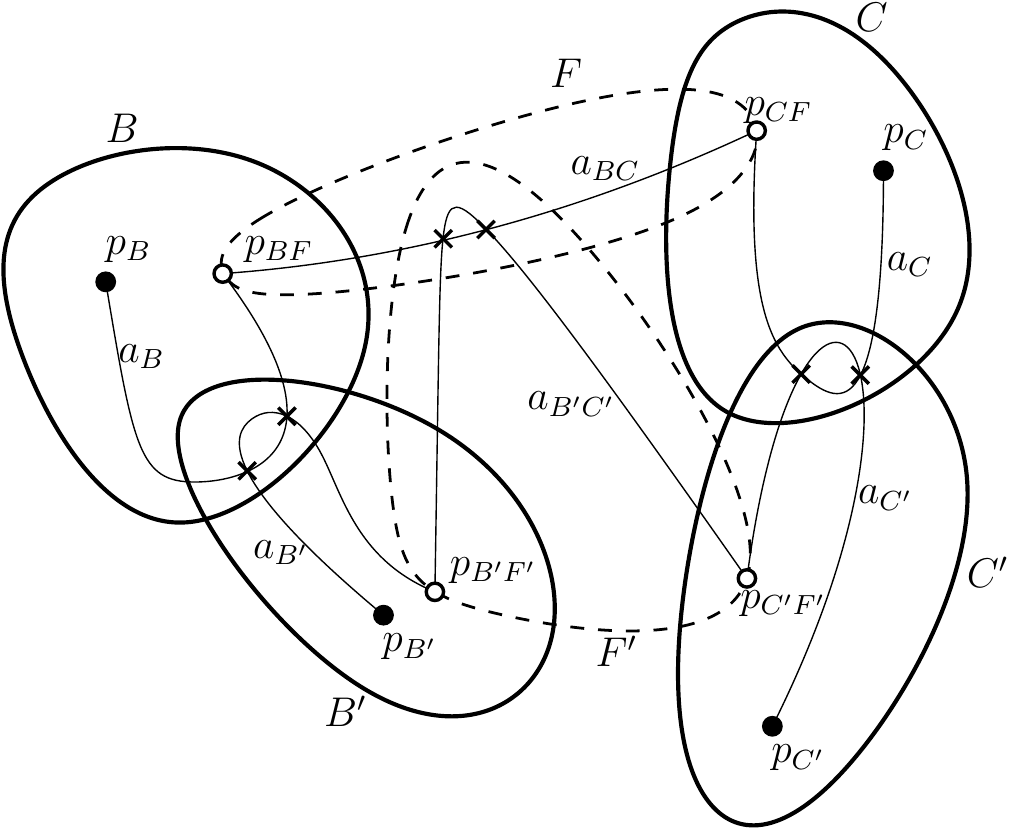}
		\caption{The edges $e_f$ and $e_{f'}$ intersect an even number of times.}
		\label{fig:psdiskownpoint}	
	\end{figure}

	If $v_B$ and $v_C$ are
	connected by an edge $f$ in $G$, then $e_f$, the drawing of the edge $f$, connecting $p_B$ and $p_C$, is as follows.
	Let $F\in \F$ be a pseudo-disk corresponding to $f$. Draw an arc $a_B$ inside $B$ from $p_B$ to a point $p_{BF}\in B\cap F$ ($p_{BF}\in C$ is allowed, also $p_{BF}$ may coincide with $p_B$). Similarly draw an arc $a_C$ inside $C$ from $p_C$ to a point $p_{CF}\in C\cap F$ ($p_{CF}\in B$ is allowed and $p_{CF}$ may coincide with $p_C$). Finally, draw an arc $a_{BC}$ inside $F$ from $p_{BF}$ to $p_{CF}$ ($p_{BF}$ and $p_{CF}$ may also coincide in which case $a_{BC}$ is of length zero). The concatenation of these three arcs is the drawing $e_f$ of the edge $f$ between the points $p_B$ and $p_C$. Note that $e_f$ may have self-crossings.
	
	We are left to prove that in this drawing of $G$ every pair of edges that do not share a vertex cross an even number of times.
	
	Let $B,C,B',C'\in \B$ with edges $f$ defined by $F$ between $v_B$ and $v_{C}$ and $f'$ defined by $F'$ between $v_{B'}$ and $v_{C'}$ 
	Suppose $e_f = a_B\cup a_{BC}\cup a_{C}$ and $e_{f'} = a_{B'}\cup a_{B'C'}\cup a_{C'}$ are the drawings of the two edges.
	Notice that the two endpoints of $a_B$ are in $B\setminus {B'}$ and the two endpoints of $a_{B'}$ are in $B'\setminus B$. Thus using Lemma \ref{lem:disjointedges} we get that $a_B$ with $a_{B'}$ intersects an even number of times. The same way we get that $a_B$ with $a_{C'}$, $a_{C}$ with $a_{B'}$ and $a_{C}$ with $a_{C'}$ intersect an even number of times. As $a_{BC}$ (resp.  $a_{B'C'}$) is disjoint from $B'$ and $C'$ (resp. $B$ and $C$), there is no intersection between $a_{BC}$ and $a_{B'},a_{C'}$ nor between $a_{B'C'}$ and $a_B,a_C$. Finally, the two endpoints of $a_{BC}$ are in $F\setminus F'$ and the two endpoints of $a_{B'C'}$ are in $F'\setminus F$ thus again using Lemma \ref{lem:disjointedges} we get that $a_{BC}$ and $a_{B'C'}$ intersect an even number of times. These together imply that indeed $e_f$ and $e_{f'}$ intersect an even number of times, as required.
\end{proof}

We note that Pach and Sharir \cite{pachsharir2} proved that (among others) pseudo-disks have linear union complexity using a similar approach as the proof of
Lemma \ref{lem:psdiskownpointdelplanar}, connecting own-points of intersecting regions along their boundaries.

\begin{corollary}\label{cor:romstrong}
	Given a family $\F$ of pseudo-disks and a finite family $\B$ of pseudo-disks, the restricted Delaunay-graph of $\B$ wrt. $\F$ is a planar graph.
\end{corollary}

\begin{proof}
    We can delete a member of $\B$ from $\B$ which corresponds to a degree-$0$ vertex in the restricted Delaunay-graph and then the restricted Delaunay-graph of the new family contains the original restricted Delaunay-graph as a subgraph apart from the degree-$0$ vertex (which does not alter planarity). We can keep doing this until possible, thus we can assume that there are no degree-$0$ vertices. In this case for every $B\in \B$ we have a point $p_B\in B$ which is in no other $C\in \B$. Indeed, taking an edge $f$ of $G$ incident to $v_B$ and the corresponding $F\in \F$, by definition of the restricted Delaunay-graph every point of $B\cap F$ is a point which is in $B$ but in no other $C\in \B$. Thus we can apply Lemma \ref{lem:psdiskownpointdelplanar} to conclude that the Delaunay-graph (and so the restricted-Delaunay-graph as well) is planar.
\end{proof}

\begin{obs}\label{obs:restricteddel}
	If $\F$ is shrinking-closed (for $\B$), then for every $F\in \F$ the corresponding hyperedge $H=\{v_B:B\cap F\ne \emptyset\}$ of $\I(\B,\F)$ either contains an edge of the restricted Delaunay-graph or $F$ contains a point contained in at least $2$ members of $\B$.
\end{obs}

Corollary \ref{cor:romstrong} shows that Lemma \ref{lem:psdiskownpointdelplanar} takes care of the planarity of the restricted Delaunay-graph. From Observation \ref{obs:restricteddel} we may think that we are left to take care of hyperedges that contain a point that is contained in at least $2$ members of $\B$. This intuition turns out to be good, in all of our main results Lemma \ref{lem:psdiskownpointdelplanar} will essentially reduce the problem to regarding the intersection hypergraph of $\B$ wrt. all points of the plane (instead of $\B$ wrt. $\F$).

Before starting the proof of Theorem \ref{thm:psdiskdelplanar} we make some more preparations:

\begin{defi}
	Given a family of regions $\B$, a point is \textbf{$k$-deep} if it is contained in exactly $k$ members of $\B$. We denote by $\partial B$ the boundary of some region $B$ of $\B$. We call a point which is in $B$ but no other $C\in \B$, an \textbf{own-point} of $B$.
\end{defi}

\begin{defi}
	A hypergraph $\HH'$ \textbf{supports} another hypergraph $\HH$ if they are on the same vertex set and for every hyperedge $H\in \HH$ there exists a hyperedge $H'\in \HH'$ such that $H'\subseteq H$.
\end{defi}

\begin{obs}\label{obs:supportinherit}
	If $\HH''$ supports $\HH'$ and $\HH'$ supports $\HH$, then $\HH''$ supports $\HH$.
\end{obs}

\begin{obs}\label{obs:supportdel}
	If a hypergraph $\HH'$ supports another hypergraph $\HH$, then the Delaunay-graph of $\HH$ is a subgraph of the Delaunay-graph of $\HH'$.
\end{obs}

Our last tool is the following theorem of Snoeyink and Hershberger (we write here a special case of what they called as the Sweeping theorem):

\begin{thm}[\protect\cite{snoeyink}]\label{thm:snoeyink}
	Let $\Gamma$ be a finite set of bi-infinite curves in the plane such that any pair of them intersects at most twice. Let $d$ be a closed curve which intersects at most twice every curve in $\Gamma$. We can sweep $d$ such that every member of the sweeping of $d$ intersects at most twice every other curve of $\Gamma$.
\end{thm}

A sweeping of $d$ in Theorem \ref{thm:snoeyink} is defined as a family $d(t)$, $t\in (-1,1]$, of pairwise disjoint curves such that
$d(0) = d$ and their union contains all points of the plane ($d(-1)$ is a degenerate curve consisting of a singular point).

\begin{proof}[Proof of Theorem \ref{thm:psdiskdelplanar}]
	
	Using Corollary \ref{cor:shrinking-closed} we can assume that $\F$ is shrinking-closed (for $\B$) and in particular, for every point $p$ which is in some $F\in \F$, $H_p=\{v_B:p\in B\in \B\}$ is a hyperedge of $\I(\B,\F)$ if it is of size at least two.
	
	We can keep deleting members of $\B$ from $\B$ which correspond to a vertex with degree $0$ or $1$ in the Delaunay-graph, during this process the Delaunay-graph of the new family keeps containing the graph induced by the original Delaunay-graph on the new (reduced) vertex set. If we can prove that the new Delaunay-graph is planar, then adding back the degree $0$ and $1$ vertices in reverse order we see that the original Delaunay-graph is also planar. Thus we can assume that every vertex of the Delaunay-graph has degree at least two. In this case for every $B\in \B$ we have a point in $B$ which is in at most one other $C\in \B$ (as there are no degree-$0$ vertices) and there are no two regions $B,C\in \B$ such that $B\subset C$ (as there are no degree-$0$ or degree-$1$ vertices).

	
	Given $\B$ and $\F$ we will modify $\B$ such that for the new family $\hat \B$ of pseudo-disks every $B\in \hat\B $ contains an own-point and furthermore we do this in a way that $\I(\hat \B,\F)$ supports $\I(\B,\F)$. Using Lemma \ref{lem:psdiskownpointdelplanar} we get that the Delaunay-graph of $\hat \B$ wrt. $\F$ is planar and then by Observation \ref{obs:supportdel} we get that the Delaunay graph of $\B$ wrt. $\F$ is also planar.
	
	We do this modification by repeating the below defined operation finitely many times, at each time decreasing by at least one the number of members of $\B$ without an own-point. We now define the three steps of this operation.
	\begin{itemize}
		\item{Step 1 - Preparation.}
		
		Take an arbitrary  $B\in \B$ that does not contain an own-point. We take a $C\in\B$ for which $v_B$ is connected to $v_C$ in the Delaunay-graph. Thus, there is a point $p$ which is in $B\cap C$ but no other member of $\B$. We morph the plane such that $B\cap C$ becomes a square $R$ such that the two intersection points of the boundary of $B$ and $C$ are on the horizontal halving line of the square (the upper part of $\partial R$ belongs to $\partial B$ and the lower part of $\partial R$ belongs to $\partial C$) and no member of $\B$ (and $\F$) intersects the vertical sides of the square. This morphing is easily doable and it leaves intact the intersection structure of $\B$ and $\F$. Denote by $p_l$ and $p_r$ the intersection points of $\partial B$ and $\partial C$, that is, the midpoints of the left and right side of $R$. This finishes the Preparation Step of our operation. See the left side of Figure \ref{fig:pullapart} for what we get after this step.

	\begin{figure}[t]
		\centering
		\subfigure[]{\includegraphics[height=8cm]{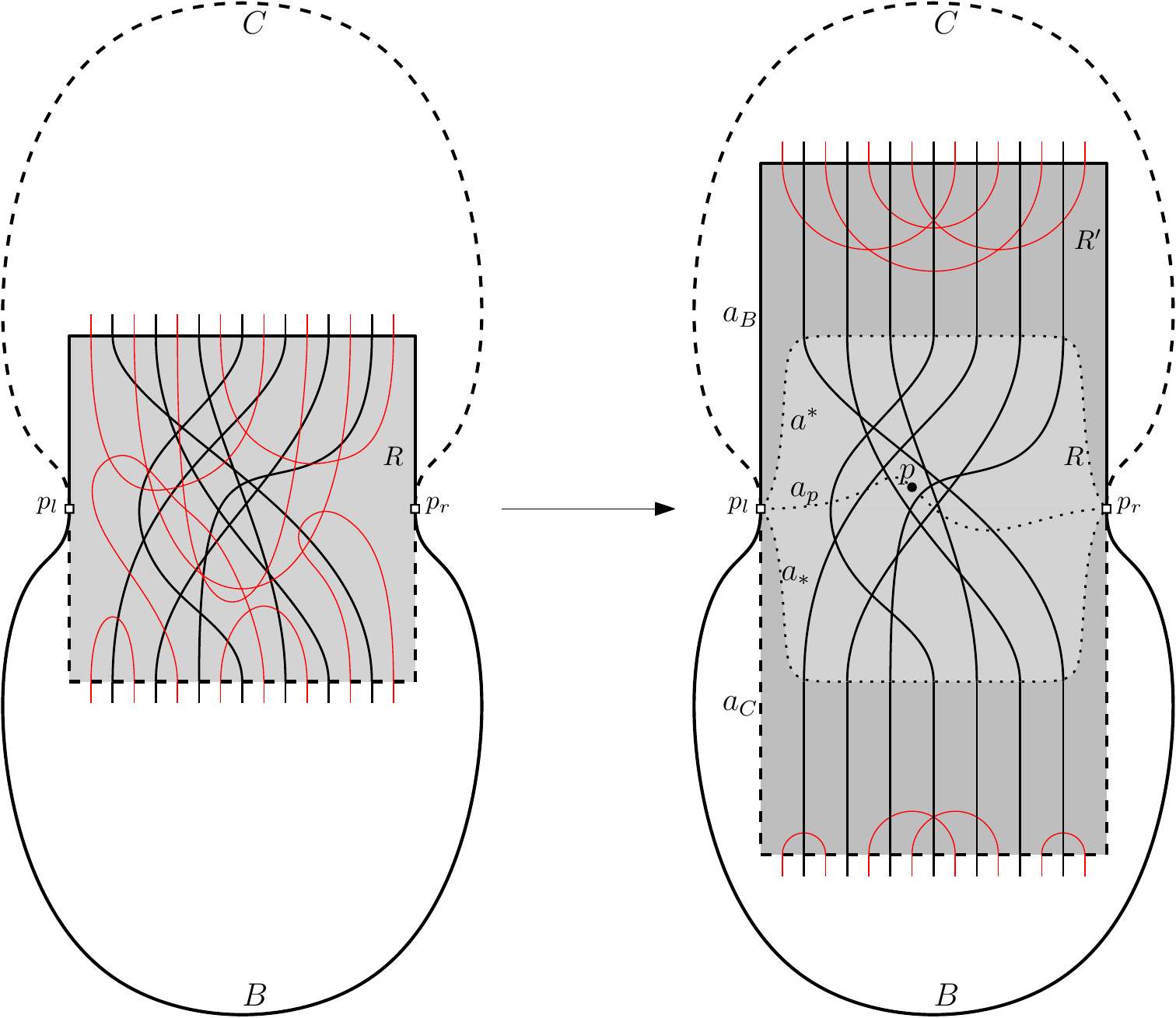}
			\label{fig:pullapart}
			\hspace{1cm}}
		\subfigure[]{\includegraphics[height=8cm]{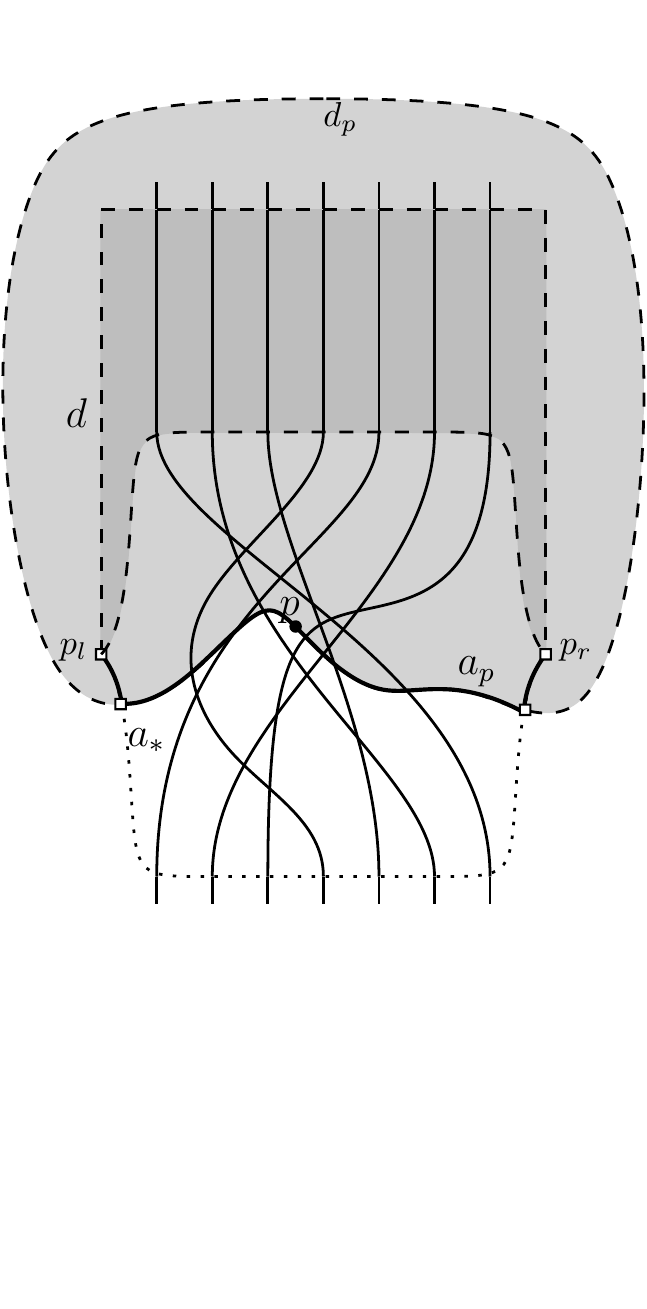}
			\label{fig:pullapart_ap}
		}
		\caption{Major steps of the operation used in the proof of Theorem \ref{thm:psdiskdelplanar}: (a) pulling apart $B\cap C$, (b) drawing $a_p$.}
	\end{figure}

		
		\item{Step 2 - Pulling apart $B\cap C$.}
		
		Now we define the Pulling apart $B\cap C$ Step of the operation which keeps the intersection structure of $\B$ and $\F$ outside $B\cap C$ intact. First we morph the plane such that the square's height is doubled to get the rectangle $R'$ while its horizontal halving line does not change (the part of $\partial R'$ above the halving line is part of $\partial B$ and the part below is part of $\partial C$) and the drawing inside $R$ remains untouched. We do this such that the intersections of members of $\B$ (and of $\F$) with the horizontal sides of the square are stretched to become vertical lines in $R'\setminus R$. Next, for every $D\in \B\setminus \{B\}$ which intersects exactly one (horizontal) side of $R'$, we redraw the part of $\partial D$ inside $B$ to be a half-circle inside $B$ that has the same endpoints. See the right side of Figure \ref{fig:pullapart}, where  these redrawn boundary parts are drawn with thin red strokes. We get a modified family $B'$, while the topology of $\F$ is unmodifed.
		
		It is easy to see that we modified $\B$ in a way that $\B'$ remains a pseudo-disk family. Next we show that $\I(\B',\F)$ supports $\I(\B,\F)$, that is for every $H\in \I(\B,\F)$ there exists a $H'\in \I(\B',\F)$ such that $H'\subseteq H$.
		If $F\in \F$ contains only depth-$1$ points, then it is disjoint from $B\cap C$ and so the hyperedge $H_F$ remains in the intersection hypergraph after pulling apart $B\cap C$. Otherwise, $F$ contains a depth-$2$ point $q$ and then $H_F$ contains the hyperedge $H_q$. So it is enough to prove that for every $H_q$ there is a hyperedge $H'$ (also corresponding to some point) in $\I(\B',\F)$ such that $H'\subseteq H_q$. For $q\notin B\cap C$, $H_q$ remains in the intersection hypergraph after pulling apart $B\cap C$. Finally, for $q\in B\cap C$ the hyperedge $H_q$ contains the hyperedge $\{v_B,v_C\}$ which is exactly the hyperedge corresponding to $p$ in $\B'$ (recall that this is the point that was only in $B$ and $C$ and no other member of $\B$, and this property remains true for $\B'$ after the operation).
		
		\item{Step 3 - Shrinking of $C$.}
		
		Now we do the final step, the Shrinking of $C$ Step of the operation.
		Let arc $a_B$ (resp. $a_C$) be the part of $\partial R'$ above (resp. below) the halving line, which arc is also part of $\partial B$ (resp. $\partial C$) after the first two steps of the operation.
		Let arc $a^*$ (resp. arc $a_*$) be the part of $\partial R$ which is above (resp. below) the halving line. We perturb the vertical parts of $a^*$ and $a_*$ slightly so that they do not overlap (nor intersect) with $a_B$ and $a_C$. See right side of Figure \ref{fig:pullapart} for an illustration.
		
		If either $a^*$ or $a_*$ contains a point $p$ which is $2$-deep (thus contained by $B$ and $C$ and no other member of $\B'$), then we redraw  $\partial C$ such that we change $a_C$ to $a^*$ or $a_*$ (whichever contains the $2$-deep point $p$) and then what we get remains a pseudo-disk family. This way the hyperedge $\{v_B,v_C\}$ is still in $\I(\B'',\F)$ as it is corresponding to a point close to $p$. Thus $\I(\B'',\F)$ supports $\I(\B',\F)$, and so $\I(\B,\F)$ as well.
		Furthermore, in $\B''$ there is a point close to $p$ which is an own-point of $B$.
		
		Unfortunately it can happen that none of $a^*$ and $a_*$ contains a $2$-deep point and so they are not suitable for redrawing $a_C$. Nevertheless, we assumed that there exists a $2$-deep point $p$ in $B$ before the operation, which remains $2$-deep after the first two steps of the operation. In the following we find an arc $a_p$ connecting $p_l$ and $p_r$ that goes inside $R$, goes through $p$ and intersects exactly once every maximal part inside $R$ of the boundary of a member of $\B'$, see Figure \ref{fig:pullapart_ap}. Assuming we have this arc $a_p$ we can redraw $\partial C$ such that we change $a_C$ to $a_p$ and what we get remains a pseudo-disk family. Also, the same way as in the previous case, we have that $\I(\B'',\F)$ supports $\I(\B,\F)$ and in $\B''$ there is a point close to $p$ which is an own-point of $B$.
		
		\item{Step 3b - Drawing $a_p$.}
		
		While the existence of such an $a_p$ is intuitively not surprising, proving its existence is a somewhat technical application of Theorem \ref{thm:snoeyink}.  
           Take the maximal parts inside $R$ of the boundaries of members of $\B'\setminus \{B,C\}$. Extending the top and bottom of these with vertical half-lines we get a family $\Gamma$ of bi-infinite curves that pairwise intersect at most twice. Let $d=a_B\cup a^*$, it intersects every curve of $\Gamma$ exactly twice. Now we can apply Theorem \ref{thm:snoeyink} to this family of curves $\Gamma$ to sweep $d$. In particular, we get a closed curve $d_p$ through $p$ which intersects every member of $\Gamma$ at most twice, see Figure \ref{fig:pullapart_ap}. Notice that it must intersect once every member of $\Gamma$ above the top side of $R$, thus it can intersect at most once (and easy to see that actually exactly once) every member of $\Gamma$ below the top side of $R$. Furthermore, $d_p$ must intersect $a_*$ at least twice. Take the maximal connected part $d'_p$ of $d_p\setminus a_*$ which contains $p$. 
        Take the curve $a_p$ which goes from $p_l$ along $a_*$ to the closest endpoint of $d'_p$ on $a_*$, continues with $d'_p$ and then goes again along $a_*$ to $p_r$. It is easy to see that $a_p$ also intersects every member of $\Gamma$ exactly once, as required.\footnote{The existence of such an $a_p$ can be proved in various ways. Instead of making the curves bi-infinite we can close them up so that they become boundaries of pseudo-disks all containing $p_l$. Then by a result of Agarwal et al. \cite{agarwal} these pseudo-disks are star-shaped (we omit the definition) which implies that there is a half-infinite curve (a `ray') starting at $p_l$ and going through $p$ intersecting all pseudo-disk boundaries once, which in turn implies the existence of the required $a_p$ as before.}
	\end{itemize}	
	
	Thus, with the above $3$-step operation we changed $\B$ such that one more member, $B$, contains an own-point. Also it is easy to see that no other member could have lost its own-point, as we have redrawn members of $\B$ only inside $B\cap C$, where there were no $1$-deep points. So after finitely many times repeating this operation we get a family $\hat \B$ in which every member has an own-point and whose intersection hypergraph $\I(\hat \B,\F)$ supports the  intersection hypergraph $\I(\B,\F)$ we started with, concluding the proof.
\end{proof}

\section{Proper coloring pseudo-disks wrt. pseudo-disks}\label{sec:psdisk}

Note that Observation \ref{obs:delaunay} can be rephrased equivalently such that if the Delaunay-graph of $\HH$ supports $\HH$, then a proper coloring of the Delaunay-graph is a proper coloring of $\HH$. More generally, it is true that:

\begin{obs}\label{obs:supportcol}
	If a hypergraph $\HH'$ supports another hypergraph $\HH$, then a proper coloring of $\HH'$ is a proper coloring of $\HH$.
\end{obs}

\begin{proof}[Proof of Theorem \ref{thm:psdiskwrtpsdisk}]
	Given the pseudo-disk family $\F$ and the finite pseudo-disk family $\B$, we want to color the vertices of $\I(\B,\F)$ corresponding to the members of $\B$ such that for every $F\in \F$ the hyperedge $H_F$ in $\B$ containing exactly those $v_B$ for which $B\cap F\ne \emptyset$ is not monochromatic (assuming that $|H_F|\ge 2$). Using Corollary \ref{cor:shrinking-closed} and Observation \ref{obs:supportcol} we can assume that $\F$ is shrinking-closed (for $\B$).
	
     Taking the direct product of a constant coloring provided by Theorem \ref{thm:psdiskwrtpt} and a $4$-coloring of the restricted Delaunay-graph which is planar by Corollary \ref{cor:romstrong}, it is not hard to see that we get a required coloring with a constant number of colors (this idea appears also in the proof of Theorem \ref{thm:linunionwrtpsdisk}). Furthermore, as we mentioned in the introduction, even the existence of a proper $6$ coloring follows immediately from Theorem \ref{thm:psdiskdelplanar} using known frameworks. However, with the operation introduced in the proof of Theorem \ref{thm:psdiskdelplanar} in hand we can show even the existence of a proper $4$ coloring, which is the optimal number of colors.
    
	We prove the existence of a $4$-coloring by induction first on the size of $\B$,  second on the size of the largest containment-minimal hyperedges and third on the number of largest containment-minimal hyperedges. We say that a hypergraph $\HH'$ is \textit{better} than a hypergraph $\HH$ (on the same vertex set) and that $\HH$ is \textit{worse} than $\HH'$ if either the size of the largest containment-minimal hyperedges is smaller in $\HH'$ than in $\HH$ or they are of the same size but there are more of them in $\HH$. 
	
	If $\B$ contains two pseudo-disks $B,C$ such that $B\subset C$, then we can proper color $\I(\B\setminus\{B\},\F)$ by induction and then color $v_B$ with a color different from the color of $v_C$  to get a proper coloring of $\I(\B,\F)$ as every hyperedge which contains $v_B$ must contain $v_C$ as well. Thus, we can assume that no two pseudo-disks in $\B$ contain one another.
	
	 We can assume that $\F$ is shrinking-closed (for $\B$) as adding non-maximal hyperedges cannot make a hypergraph worse (and also cannot increase its chromatic number). In particular, we can assume that for every point $p$ which is in some $F\in \F$, $\I(\B,\F)$ contains a hyperedge $H_p=\{v_B:p\in B\in \B\}$ (if it is of size at least two).
	 
	Take one of the largest containment-minimal hyperedges, $H\in \I(\B,\F)$.
	
	If $H$ is of size $2$ that means that the Delaunay-graph supports $\I(\B,\F)$ and then by Theorem \ref{thm:psdiskdelplanar} we can color it with $4$ colors, which by Observation \ref{obs:supportcol} is a proper coloring of $\I(\B,\F)$, as required. This starts the induction.

	Otherwise, $H$ has $l\ge 3$ vertices.
	Assume by induction that every intersection hypergraph of pseudo-disks wrt. pseudo-disks better than $\I(\B,\F)$ admits a proper $4$-coloring.
	
	If $H$ corresponds to some $F$ that contains only $1$-deep points, then using Corollary \ref{cor:shrink2} and that $\F$ is shrinking-closed we get that $H$ is not containment-minimal, a contradiction. Thus, $F$ contains a point which is at least $2$-deep and then using that $H$ is containment-minimal and $\F$ is shrinking-closed there must be a point $p$ (any point inside a pseudo-disk corresponding to $H$ is such) for which actually $H=H_p$. Take two members $B,C\in B$ for which $v_B,v_C\in H_p$. On $B\cap C$ we do the first two steps of the operation used in the proof of Theorem \ref{thm:psdiskdelplanar} to get a new pseudo-disk family $\B'$. As we have seen, the intersection hypergraph $\I(\B',\F)$ supports $\I(\B,\F)$ and thus $\I(\B',\F)$ is not worse than $\I(\B,\F)$. If $B\cap C$ in $\B'$ contains a $2$-deep point, then $\{v_B,v_C\}$ is a hyperedge of  $\I(\B',\F)$ and the number of size-$l$ hyperedges did actually decrease (as $H$ is not containment minimal anymore) and so $\I(\B',\F)$ is better than $\I(\B,\F)$. Then by induction $\I(\B',\F)$ can be colored with $4$ colors, which by Observation \ref{obs:supportcol} is a proper coloring of $\I(\B,\F)$ as well.
	
	The only case left is when $B\cap C$ still does not contain a $2$-deep point. In this case similarly to Step 3 of the operation in the proof of Theorem \ref{thm:psdiskdelplanar}, we redraw  $\partial C$ such that we change $a_C$ to $a^*$ (see again the right side of Figure \ref{fig:pullapart}). In the new family $\B''$ every point in $R$ is still covered at least twice, and nothing changed outside $R$, thus $\I(\B'',\F)$ supports $\I(\B',F)$ (and then in turn it suppors $\I(\B,\F)$). Moreover, the hyperedge corresponding to $p$ is a subset of $H_p$ but it does not contain $v_C$ anymore, so it is a proper subset of $H$ and is of size at least $2$. This implies that the number of size-$l$ hyperedges did decrease. Then again $\I(\B'',\F)$  is better than $\I(\B,\F)$ and then by induction can be colored with $4$ colors, which by Observation \ref{obs:supportcol} is a proper coloring of $\I(\B,\F)$ as well.
\end{proof}
Observe that already for coloring points wrt. disks we may need $4$ colors (as there are points whose intersection hypergraph wrt. disks is a complete graph on four vertices), so $4$ is also a lower bound for coloring pseudo-disks wrt. pseudo-disks.

Without going into details, we note that from the proofs we can create (low-degree) polynomial time coloring algorithms, supposing that initially we are given reasonable amount of information (e.g., the whole structure of the arrangement (all the vertices, edges and faces) of the pseudo-disks of $\B$ and also for every hyperedge $H$ in $\I(\B,\F)$ we are given an $F$ corresponding to $H$).

\section{Consequences}\label{sec:cfvc}

First we prove Corollary \ref{cor:psdiskcf} about conflict-free coloring the intersection hypergraph of pseudo-disks wrt. pseudo-disks.

\begin{proof}[Proof of Corollary \ref{cor:psdiskcf}]
    As mentioned already, the proof uses the by now standard framework of \cite{smorod-onthe}.
    Given the family $\B$ containing $n$ regions, the first step is to color it properly wrt. $\F$ using $c$ dummy colors using Theorem \ref{thm:psdiskwrtpsdisk} ($c=4$). Take the biggest color class $\C_1$ and give it color $1$. We continue with $\B_1=\B\setminus \C_1$. In the $s$'th step we start with the subfamily $\B_{s-1}$ and we color it properly wrt. $\F$ using $c$ dummy colors and take the biggest color class $\C_{s}$ and give it color $s$ and we continue the process with $\B_{s}=\B_{s-1}\setminus C_{s}$. We repeat this until we color all members of $\B$. In each step the remaining family is reduced by a ratio at least $(c-1)/c$ and so we finish in $O(\log n)$ steps. It is easy to check that this coloring is indeed conflict-free, as in every hyperedge $H_F$ defined by a region $F\in \F$ take the highest color $s$ appearing on its vertices (corresponding to members of $\B$). This color must appear only on one vertex otherwise at step $s$ the hyperedge defined by $F$ would have been a monochromatic hyperedge with at least two vertices in the dummy coloring, contradicting the fact that this dummy coloring was proper.
\end{proof}

Now we continue with the results related to the VC-dimension of the intersection hypergraph of pseudo-disks wrt. pseudo-disks. In order to do this, we assume that the reader is familiar with the definitions of VC-dimension and shattering. As we mentioned, Aronov et al. \cite{aronov18} proved that given a pseudo-disk family $\F$ and a finite subfamily $\B$ of $\F$, the intersection hypergraph $\I(\B,\F)$ has VC-dimension at most $4$ and this bound is tight. We now show that the same way Lemma \ref{lem:psdiskownpointdelplanar} implies that $\I(\B,\F)$ has VC-dimension at most $4$ when $\B$ is any finite pseudo-disk family, not necessarily a subfamily of $\F$:

\begin{proof}[Proof of Theorem \ref{thm:vcdim}]
    The proof is essentially the same as in \cite{aronov18}, except that we use Lemma \ref{lem:psdiskownpointdelplanar}. In a shattered subset $V_s$ of vertices in $\I(\B,\F)$ by definition every subset of $V_s$ is a hyperedge. Thus, for every vertex $v_B\in V_s$, $\{v_B\}$ is a hyperedge defined by some $F_B\in \F$ and so the points in $B\cap F_B$ must be own-points of $B$. Also, all pairs of vertices in $V_s$ must define a Delaunay-edge by definition of shattering, while the Delaunay-graph is planar by Lemma \ref{lem:psdiskownpointdelplanar}. As a planar complete graph has at most $4$ points, $|V_s|\le 4$, as required.
    
    Finally, the construction of \cite{buzaglo} showing that the bound was tight already when $\B$ is a subfamily of $\F$ shows that it is still tight in our more general setting.
\end{proof}

Theorem \ref{thm:ksets} follows directly from the arguments of \cite{aronov18}, except that one needs to use Theorem \ref{thm:vcdim} instead of the weaker version they used (where $\B$ was a subfamily of $\F$). The details are left to the interested reader.

\section{Coloring a family with linear union complexity wrt. pseudo-disks}\label{sec:linunion}

As in the literature in the definition of union complexity sometimes vertices and other times edges of the arrangement are counted, we first show that this does not affect the property of having linear union complexity (apart from slightly changing the constant $c$ in having $c$-linear union complexity):

\begin{lemma}\label{lem:unioncompldef}
    Given a family $\B$ of $n$ Jordan regions, whose boundary contains $e(\partial\B)$ edges and $v(\partial \B)$ vertices of the arrangement, $v(\partial\B)\le e(\partial\B)\le v(\partial\B)+n$.
\end{lemma}	

\begin{proof}
    Going clockwise around the closed Jordan curves forming $\partial \B$, the boundary of $\B$, we see that every vertex of the arrangement is followed by a (different) edge of the arrangement, hence $v(\partial\B)\le e(\partial\B)$. On the other hand every edge is preceded by a (different) vertex except the ones that alone form a closed Jordan curve, that is, that form the boundary of a single region in $\B$. There are at most $n$ such edges, thus $e(\partial\B)\le v(\partial\B)+n$.
\end{proof}

Before proving Theorem \ref{thm:linunionwrtpsdisk} we make some preparations. First we prove an easy lemma using random sampling, that is, using the Clarkson-Shor method \cite{clarksonshor}. 
We say that a face is $k$-deep if its interior points are $k$-deep, that is, the face is contained in exactly $k$ members of the family.

\begin{lemma}\label{lem:kdeep}
    In a family $\B$ of $n$ Jordan regions with linear union complexity, the number of $k$-deep faces is $O_k(n)$.
\end{lemma}

\begin{proof}	
    In the arrangement of $\B$, associate with each $k$-deep face $T$ one of the edges $e_T$ on its boundary. No edge is associated with two faces this way. 
    
    We count the number $s$ of associated edges $e_T$.
    Take a random subfamily $\B'$ of $\B$ by taking each $B\in \B$ with probability $1/2$. With probability at least $1/2^{k+3}$ an edge $e_T$ counted in $s$ ends up on $\partial \B'$, the boundary of $\B'$. Indeed, this happens whenever the at most $k$ regions that contain this edge in their interior are not in $\B'$ but the region whose boundary contains $e_T$ and two other regions whose boundaries go through the two endvertices of $e_T$ are in $\B'$. Thus $1/2^{k+3} s\le E(e(\partial \B'))\le E(c|\B'|)\le cn$ implying that $s=O_k(n)$. The second inequality followed from our assumption that $\B$ has $c$-linear union complexity.
    
    The number of $k$-deep faces equals the number $s$ of associated edges, which is $O_k(n)$, concluding the proof.
\end{proof}

It will be comfortable to use the following definition:

\begin{defi}
    A graph $G$ is \textbf{$k$-degenerate} if every subgraph of $G$ contains a point with degree at most $k$.
\end{defi}

If $G$ is $k$-degenerate for some constant $k$ we also simply say that $G$ is degenerate. For example planar graphs are $5$-degenerate. The following is a well-known property of degenerate graphs:

\begin{obs}\label{obs:degenerate}
    If a graph $G$ is $k$-degenerate, then we can proper color it with $k+1$ colors.
\end{obs}

\begin{proof}
    Take a smallest degree vertex $v$ of $G$ and proper color with $k+1$ colors by induction the vertices of the subgraph induced by the rest of the vertices. Finally, color $v$ with a color different from the color of its at most $k$ neighbors.
\end{proof}

\begin{proof}[Proof of Theorem \ref{thm:linunionwrtpsdisk}]
    
    Given the family $\B$ of $n$ regions with bounded union complexity and the family $\F$ of pseudo-disks, we want to color the vertices of $\I(\B,\F)$ corresponding to the members of $\B$ such that for every $F\in \F$ the hyperedge $H_F$ in $\B$ containing those $v_B$ for which $B\cap F\ne \emptyset$ is not monochromatic whenever it is of size at least $2$.
    
    Using Corollary \ref{cor:shrinking-closed} we can assume that $\F$ is shrinking-closed (for $\B$). In particular, for every point $p$ which is in some $F\in\F$, $H_p=\{v_B:p\in B\in \B\}$ is a hyperedge of $\I(\B,\F)$ if it is of size at least two.
    
    \smallskip
    
    First we color $\B$ wrt. points, that is, we color the vertices $v_B$ with $c_1$ many colors for some constant $c_1$ such that for an arbitrary point $p$ in the plane the hyperedge $H_p=\{v_B:p\in B\in \B\}$ is not monochromatic whenever it is of size at least $2$. Note that this is the same as Theorem \ref{thm:linunionwrtpt}, we include its proof for the sake of completeness.
    
    By Lemma \ref{lem:kdeep} there are linearly many different $H_p$'s with $|H_p|=2$, in other words the Delaunay-graph of the regions of $\B$ wrt. points has linearly many edges. Thus there is a vertex $v_B$ with degree bounded by a constant.  Color by induction the Delaunay-graph of the family we get by deleting $B$ from $\B$ with constantly many colors and then color $v_B$ with a color different from the colors of its neighbors in the Delaunay-graph of $\B$. We claim that this is a coloring in which each $H_p$ of size at least two is not monochromatic. By induction this is true for every $H_p$ that contains at least two vertices besides $v_B$. Thus only $H_p$'s which are of size exactly two and contain $v_B$ may be monochromatic at this point. However, we colored $v_B$ exactly in a way to make these hyperedges also non-monochromatic.\footnote{In order to proper color $\B$ wrt. points, it is tempting to simply proper color the Delaunay-graph by Observation \ref{obs:degenerate}. This would not be correct, as the Delaunay-graph might not support the intersection hypergraph of $\B$ wrt. points.}
    
    \smallskip
    
    Take an arbitrary $F\in \F$ such that $H_F$ is of size at least $2$.
    
    If $F\in \F$ contains a point $p$ which is in at least two members of $\B$, then $H_p$ is a subset of $H_F$ and is not monochromatic, thus $H_F$ is also not monochromatic.

    Thus, we only need to take care of the hyperedges defined by $F\in \F$ that do not contain any point $p$ that is in at least two members of $\B$. We do this by defining a second coloring using $c_2$ many colors.
    
    In this case $F\in \F$ intersects the members of $\B$ in disjoint regions. By Corollary \ref{cor:shrink2} and as $\F$ is shrinking-closed, there exists an $F'\in \F, F'\subseteq F$ such that $F'$ intersects exactly two members of $\B$. This means exactly that every such $\F$ contains as a subset an edge of the restricted Delaunay-graph $G$ of $\B$ wrt. $\F$. Next we prove that $G$ can be colored properly with constantly many colors.

    \begin{figure}
        \centering
        \includegraphics[height=8cm]{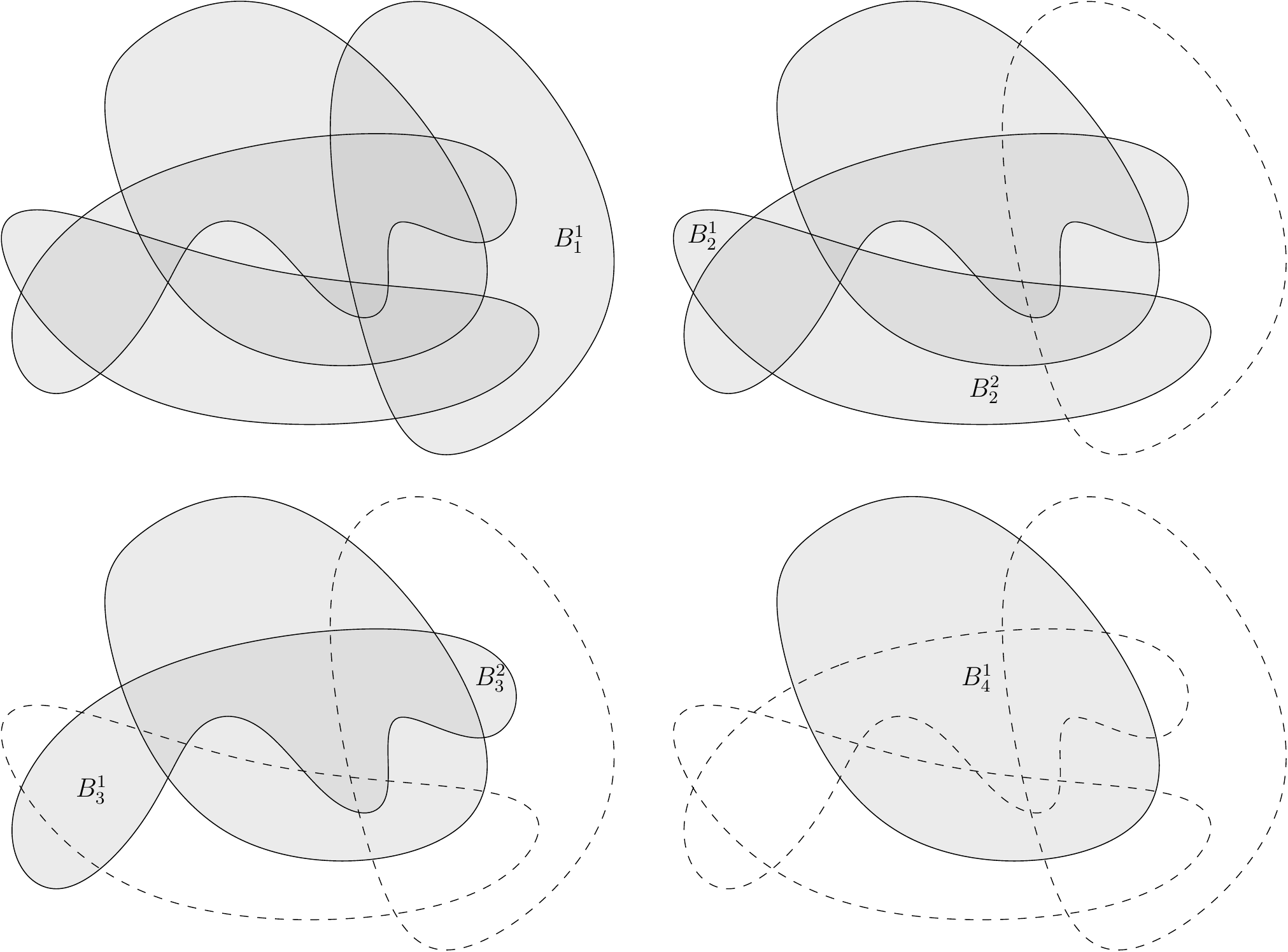}
        \caption{Steps of the construction of the disjoint subregions.}
        \label{fig:subregions}	
    \end{figure}
    
    We define subregions of the regions of $\B$ in the following way (see Figure \ref{fig:subregions} for an example).
    By Lemma \ref{lem:kdeep} there are linearly many $1$-deep faces. Thus, there is a $B_1\in B$ which contains only constantly many $1$-deep faces. Let $B_1^1,B_1^2,\dots B_1^{z_i}$ be these faces. Now we repeat this process for $\B\setminus \{B_1\}$ (and define in a $B_2\in \B\setminus \{B_1\} $ constantly many subregions, etc.) until there are no regions left. As by definition linear union complexity is hereditary, we can apply Lemma \ref{lem:kdeep} at every step and so at the end in every region we defined constantly many subregions (with the same constant), we collect these into the family $\B'$\footnote{Note that every region in $\B'$ is a face of the arrangement of the subfamily of regions of $\B$ at that moment in the process, but usually a union of more faces of the original arrangement $\B$.}. Observe that the regions in $\B'$ are connected, pairwise disjoint and their union equals the union of the regions in $\B$ (apart from the boundaries). That is, $\B'$ is a partition of the union of the regions in $\B$ (apart from the boundaries). In particular, every $1$-deep face of the arrangement $\B$ is contained in one of the regions in $\B'$.
    Notice that $\B'$ has at most constant times $n$ regions.
    
    Now we define a graph $G'$ with the same number of edges as $G$ and with at most constant times $n$ vertices. For every $B_i^j\in \B'$ it has a vertex $v_i^j$ associated with $B_i^j$. We connect two vertices $v_i^j$ and $v_{i'}^{j'}$ whenever there exists an $F\in \F$ which intersects exactly two members of $\B$, $B_i$ and $B_{i'}$ and furthermore $F\cap B_i\subseteq B_i^j$ and $F\cap B_{i'}\subseteq B_{i'}^{j'}$. Notice that $G'$ is a subgraph of the restricted Delaunay-graph of $\B'$ wrt. $\F$ and as $\B'$ is a family of disjoint connected regions, by Lemma \ref{lem:romdisjoint} this is a planar graph.	
    
    As $G'$ is planar, it has at most $3$ times as many edges as vertices and so it has at most constant times $n$ edges. By definition $G$ and $G'$ have the same number of edges, associated with members $F\in \F$ which intersect exactly two members of $\B$ with disjoint intersections. Thus $G$ also has linearly many edges and so must have a vertex with degree bounded by a constant. As we can apply this argument for any subfamily of $\B$ and any induced subgraph of $G$ is a subgraph of the Delaunay-graph of the family of corresponding members of $\B$, we get that $G$ is degenerate. Thus by Observation \ref{obs:degenerate} it admits a proper coloring with $c_2$ many colors for some constant $c_2$.
    
    The direct product of the above two colorings with $c_1$ and $c_2$ many colors is a proper $c_1c_2$-coloring of the hypergraph, concluding the proof.
\end{proof}

\begin{figure}
    \centering
    \includegraphics[height=6cm]{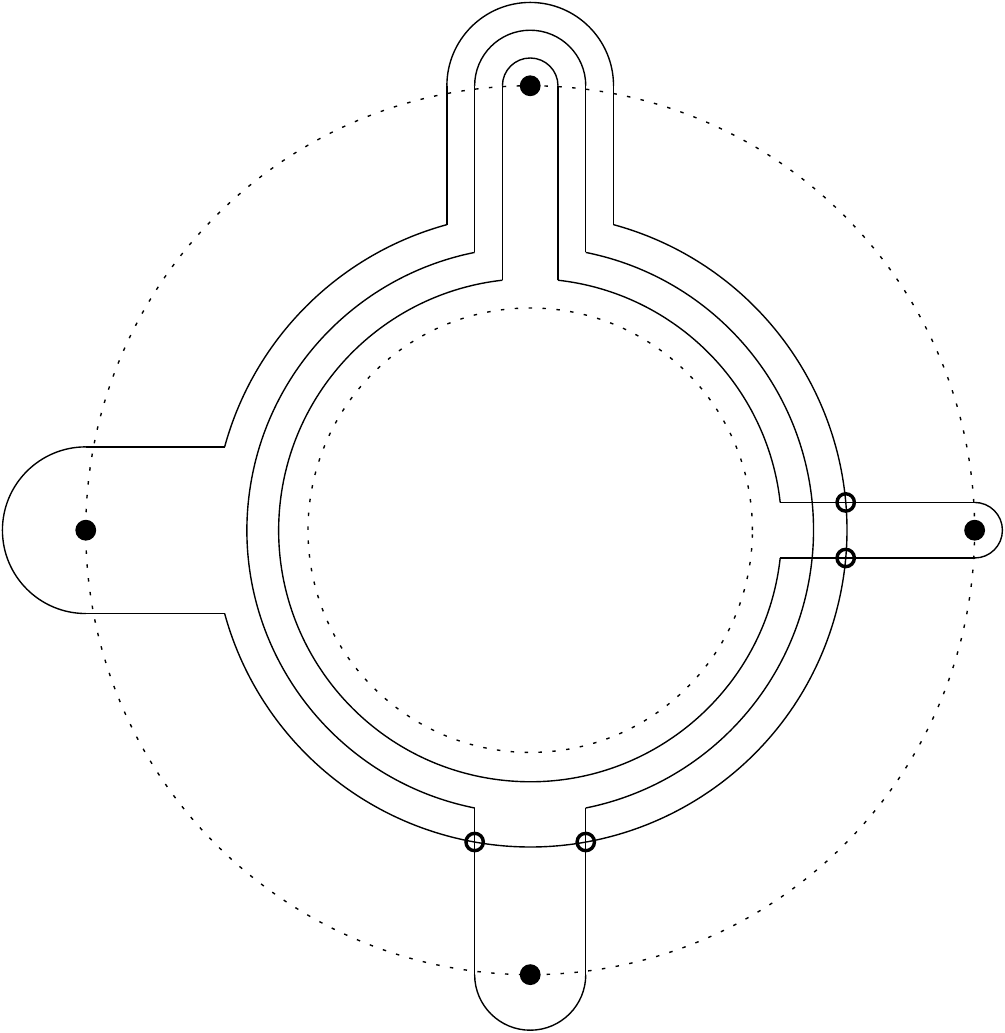}
    \caption{Construction with $n=4$ points, $3$ of the $4\choose 2$ regions are shown, their union complexity is $4$.}
    \label{fig:bearears}	
\end{figure}

\begin{proof}[Proof of Corollary \ref{cor:linunioncf}]
    The proof is exactly the same as of Corollary \ref{cor:linunioncf} just instead of Theorem \ref{thm:psdiskwrtpsdisk} we need to use Theorem \ref{thm:linunionwrtpsdisk}.
\end{proof}

We note again that the above proof of Theorem \ref{thm:linunionwrtpsdisk}  works without modification also when instead of Jordan-regions (i.e., bounded regions whose boundary is a single closed Jordan curve), $\B$ is a family of bounded regions whose boundary is a union of a finite number of disjoint Jordan-curves (in particular the regions are not necessarily simple or connected).

Without going into details, we also note again that from the proofs we can create efficient coloring algorithms. 

Now we show that we cannot always color properly points wrt. a family of linear union complexity with constantly many colors and so in particular in Theorem \ref{thm:linunionwrtpsdisk} $\F$ cannot be changed to be a family of linear union complexity.

\begin{claim}
    For every $n$ there exists a set $S$ of $n$ points and a family $\F$ of $n\choose 2$ simply connected regions with linear union complexity such that $\I(S ,\F)$ is the complete graph on $n$ vertices. More precisely, for every pair of points in $S$ there is a region in $\F$ that contains exactly these two points, every pair of regions intersects at most four times and for every $m$ the union complexity of any subfamily of size $m$ is at most $4m-4$.
\end{claim}
\begin{proof}
    The construction is the following: distribute evenly $n$ points $S=\{p_0,\dots p_{n-1}
    \}$ on a circle of radius $2$ around the origin $o$.
    Choose an $\epsilon$ small enough.
    Let $F_{\{i,j\}}$ be the region we get by taking the union of the following three regions: a disk of radius $1+(in+j)\epsilon$, the Minkowski sum of the segment $op_i$ and a disk of radius $(in+j)\epsilon$ and the Minkowski sum of the segment $op_j$ and a disk of radius $(in+j)\epsilon$. See Figure \ref{fig:bearears}.

    Notice that $F_{\{i,j\}}$ contains $p_i$ and $p_j$ but no other points of $S$ if $\epsilon$ is small enough, thus $\I(S ,\F)$ is the complete graph, as required.
    
    It is easy to see that two regions can intersect at most $4$ times. For an arbitrary subfamily $\F'$ of $m$ regions take the region $F^0=F_{\{i,j\}}\in \F'$ for which $in+j$ is largest possible. By the construction every intersection point of boundaries that are on the boundary of $\cup\F'$ is on the boundary of $F^0$. As $F^0$ intersects all other regions in at most $4$ points, we get the upper bound $4m-4$ on the union complexity of $\F'$. 
    
    We remark that it is not much harder to show that $\min(2n-4,4m-4)$ is also an upper bound on the complexity of $\F'$.
\end{proof}

\section{Discussion}\label{sec:discussion}
As mentioned, the primary motivation behind considering such problems lies in applications to conflict-free colorings and to cover-decomposability problems.

When the main interest is in conflict-free colorings, the $\log n$ factor makes it less interesting to optimize the bound on the proper coloring result. However, when considering the relation to cover-decomposability problems, finding the exact value is of great interest. This makes it important that in Theorem \ref{thm:psdiskwrtpsdisk} we could prove an exact upper bound for such a wide class of intersection hypergraphs. On the other hand, in Theorem \ref{thm:linunionwrtpsdisk} we aimed for keeping the proof as simple as possible, and so we did not optimize the number of colors.

In cover-decomposability problems the typical question asks if for some $m$ we can properly $2$-color $\I_{m}(\B,\F)$, the subhypergraph of $\I(\B,\F)$ containing only hyperedges of size at least $m$. If we properly color $\I_{m}(\B,\F)$, then we say that we color $\B$ wrt. $\F$ for $m$.

Usually either $\B$ or $\F$ is the family of points of the plane. Originally researchers concentrated on the problem when the other family, $\F$ or $\B$, is the family of translates of a convex region (see, e.g., \cite{surveycd} for a history of this research direction).

In the past years researchers started to concentrate more on problems where $\B$ or $\F$ is the family of homothets of a convex region. One of the the most intriguing open problems is whether we can $2$-color points wrt. homothets of a given convex polygon for some constant $m$. The existence of a $2$-coloring for some $m$ was proved for triangles \cite{octant} and the square \cite{homotsquare} and recently for convex polygons the existence of a $3$-coloring for some $m$ was proved \cite{3propercol}. On the other hand, for coloring points wrt. disks $2$ colors are not enough (for any $m$) \cite{indec}. As we have seen already in the introduction, a $4$-coloring of points wrt. disks always exists (for $m=2$), however the existence of a $3$-coloring (for some $m$) is an open problem (see also the remark after Problem \ref{prob:general}).

For the dual problem, about coloring the homothets of a convex polygon wrt. points, while proper $2$-coloring exists for the homothets of triangles \cite{octant} for some $m$, for the homothets of any other convex polygon there is not always a $2$-coloring \cite{kovacs} (for any $m$). Also, for coloring disks wrt. points $2$ colors are not enough (for any $m$) \cite{unsplittable}.

In \cite{3propercol} it was conjectured that points wrt. pseudo-disks and pseudo-disks wrt. points can be proper $3$-colored (for some $m$). Encouraged by Theorem \ref{thm:psdiskwrtpsdisk} we pose the common generalization of these conjectures:

\begin{problem}\label{prob:general}
	Does there exist a constant $m$ such that given a family $\F$ of pseudo-disks and a finite family $\B$ of pseudo-disks, $\I_m(\B,\F)$ always admits a proper coloring with $3$ colors?
\end{problem}

We note that T\'oth \cite{gezapc} showed an example that choosing $m=3$ is not enough already when $\B$ is a set of points and $\F$ is a family of pseudo-disks.

The case when $\B$ is a family of points and $\F$ is the family of disks was asked by the author of this paper more than $10$ years ago \cite{wcf1,wcf2} and is still open (in this special case even $m=3$ might be true).

As we mentioned in the introduction, one can easily change in the definition of the incidence hypergraph the incidence relation to the containment or reverse-containment relation. Thus we can define the following hypergraphs: $\HH^\subset (\B,\F)$ contains a hyperedge $H=\{v_B:B\subset F\}$ for every $F\in \F$; similarly, $\HH^\supset (\B,\F)$ contains a hyperedge $H=\{v_B:B\supset F\}$ for every $F\in \F$ (if these sets are of size at least two).

Thus we can ask if the respective variants of Theorem \ref{thm:psdiskwrtpsdisk} and Problem \ref{prob:general} hold for these hypergraphs:

\begin{problem}
	Given a family $\F$ of pseudo-disks and a finite family $\B$ of pseudo-disks, does $\HH^\subset_m (\B,\F)$ (resp. $\HH^\supset_m (\B,\F)$) always admit a proper coloring with $4$ colors?
	
	Does there exist a constant $m$ such that $\HH^\subset_m (\B,\F)$ (resp. $\HH^\supset_m (\B,\F)$) always admits a proper coloring with $3$ colors?
\end{problem}

The variant of this problem where $\F$ and $\B$ are families of intervals on the line was regarded in \cite{oktans9} where among others it was shown that for intervals these two classes (the class of containment and reverse-containment hypergraphs of intervals) are the same.

Finally we mention that similarly to pseudo-disks, pseudo-halfplanes are natural generalizations of halfplanes. For pseudo-halfplanes it was also possible to reprove the same coloring results that are known for halfplanes \cite{abafree}. Also, due to the lack of direct relation to pseudo-disks, we did not mention axis-parallel rectangles, whose intersection hypergraph coloring problems are some of the most interesting open problems of the area.

\subsubsection*{Acknowledgment}

    This paper is dedicated to my grandfather who died the same day when I found the proof for Theorem \ref{thm:psdiskwrtpsdisk} (at that time with a non-explicit constant).
    Thanks to Dömötör Pálvölgyi for the inspiring discussions about these and other problems about coloring geometric hypergraphs defined by pseudo-disks. Thanks to Chaya Keller 
    for many useful remarks about the realization of this paper.

\bibliographystyle{plainurl}
\bibliography{psdisk}

\begin{thebibliography}{10}

\bibitem{homotsquare}
Eyal Ackerman, Bal{\'{a}}zs Keszegh, and Mate Vizer.
\newblock Coloring points with respect to squares.
\newblock {\em Discrete {\&} Computational Geometry}, jun 2017.

\bibitem{agarwal}
Pankaj~K Agarwal, Eran Nevo, J{\'a}nos Pach, Rom Pinchasi, Micha Sharir, and
  Shakhar Smorodinsky.
\newblock Lenses in arrangements of pseudo-circles and their applications.
\newblock {\em Journal of the ACM (JACM)}, 51(2):139--186, 2004.

\bibitem{aronov18}
B.~{Aronov}, A.~{Donakonda}, E.~{Ezra}, and R.~{Pinchasi}.
\newblock {On Pseudo-disk Hypergraphs}.
\newblock {\em ArXiv e-prints}, February 2018.
\newblock \href {http://arxiv.org/abs/1802.08799} {\path{arXiv:1802.08799}}.

\bibitem{buzaglo}
Sarit Buzaglo, Rom Pinchasi, and G{\"{u}}nter Rote.
\newblock Topological hypergraphs.
\newblock In {\em Thirty Essays on Geometric Graph Theory}, pages 71--81.
  Springer New York, oct 2012.

\bibitem{cardinalkorman}
Jean Cardinal and Matias Korman.
\newblock Coloring planar homothets and three-dimensional hypergraphs.
\newblock {\em Computational geometry}, 46(9):1027--1035, 2013.

\bibitem{pseudodisjoint}
Timothy~M. Chan and Sariel Har{-}Peled.
\newblock Approximation algorithms for maximum independent set of pseudo-disks.
\newblock {\em Discrete {\&} Computational Geometry}, 48(2):373--392, 2012.

\bibitem{hanani1}
Chaim Chojnacki.
\newblock {\"U}ber wesentlich unpl{\"a}ttbare kurven im dreidimensionalen
  raume.
\newblock {\em Fundamenta Mathematicae}, 23(1):135--142, 1934.

\bibitem{clarksonshor}
Kenneth~L Clarkson and Peter~W Shor.
\newblock Applications of random sampling in computational geometry, ii.
\newblock {\em Discrete \& Computational Geometry}, 4(5):387--421, 1989.

\bibitem{evenlotker}
Guy Even, Zvi Lotker, Dana Ron, and Shakhar Smorodinsky.
\newblock Conflict-free colorings of simple geometric regions with applications
  to frequency assignment in cellular networks.
\newblock {\em {SIAM} Journal on Computing}, 33(1):94--136, jan 2003.

\bibitem{fekete-int}
S.~P. {Fekete} and P.~{Keldenich}.
\newblock {Conflict-Free Coloring of Intersection Graphs}.
\newblock {\em ArXiv e-prints}, September 2017.
\newblock \href {http://arxiv.org/abs/1709.03876} {\path{arXiv:1709.03876}}.

\bibitem{pseudoerdosszekeres}
Andreas~F. {Holmsen}, Hossein {Nassajian Mojarrad}, János {Pach}, and Gábor
  {Tardos}.
\newblock {Two extensions of the Erd\H os-Szekeres problem}.
\newblock {\em ArXiv e-prints}, October 2017.
\newblock \href {http://arxiv.org/abs/1710.11415} {\path{arXiv:1710.11415}}.

\bibitem{Kedem1986}
Klara Kedem, Ron Livne, J{\'a}nos Pach, and Micha Sharir.
\newblock On the union of jordan regions and collision-free translational
  motion amidst polygonal obstacles.
\newblock {\em Discrete {\&} Computational Geometry}, 1(1):59--71, Mar 1986.

\bibitem{smorod-int}
C.~{Keller} and S.~{Smorodinsky}.
\newblock {Conflict-Free Coloring of Intersection Graphs of Geometric Objects}.
\newblock {\em ArXiv e-prints}, April 2017.
\newblock \href {http://arxiv.org/abs/1704.02018} {\path{arXiv:1704.02018}}.

\bibitem{wcf1}
Bal{\'{a}}zs Keszegh.
\newblock Weak conflict-free colorings of point sets and simple regions.
\newblock In Prosenjit Bose, editor, {\em Proceedings of the 19th Annual
  Canadian Conference on Computational Geometry, {CCCG} 2007, August 20-22,
  2007, Carleton University, Ottawa, Canada}, pages 97--100. Carleton
  University, Ottawa, Canada, 2007.

\bibitem{wcf2}
Bal{\'{a}}zs Keszegh.
\newblock Coloring half-planes and bottomless rectangles.
\newblock {\em Computational geometry}, 45(9):495--507, 2012.

\bibitem{octant}
Bal{\'{a}}zs Keszegh and D{\"{o}}m{\"{o}}t{\"{o}}r P{\'{a}}lv{\"{o}}lgyi.
\newblock Octants are cover-decomposable.
\newblock {\em Discrete {\&} Computational Geometry}, 47(3):598--609, 2012.

\bibitem{abafree}
Bal{\'{a}}zs Keszegh and D{\"{o}}m{\"{o}}t{\"{o}}r P{\'{a}}lv{\"{o}}lgyi.
\newblock An abstract approach to polychromatic coloring: Shallow hitting sets
  in aba-free hypergraphs and pseudohalfplanes.
\newblock In Ernst~W. Mayr, editor, {\em Graph-Theoretic Concepts in Computer
  Science - 41st International Workshop, {WG} 2015, Garching, Germany, June
  17-19, 2015, Revised Papers}, volume 9224 of {\em Lecture Notes in Computer
  Science}, pages 266--280. Springer, 2015.

\bibitem{oktans9}
Bal{\'{a}}zs Keszegh and D{\"{o}}m{\"{o}}t{\"{o}}r P{\'{a}}lv{\"{o}}lgyi.
\newblock More on decomposing coverings by octants.
\newblock {\em Journal of Computational Geometry}, 6(1):300--315, 2015.

\bibitem{3propercol}
Bal{\'{a}}zs Keszegh and D{\"{o}}m{\"{o}}t{\"{o}}r P{\'{a}}lv{\"{o}}lgyi.
\newblock Proper coloring of geometric hypergraphs.
\newblock In {\em Symposium on Computational Geometry}, volume~77 of {\em
  LIPIcs}, pages 47:1--47:15. Schloss Dagstuhl - Leibniz-Zentrum fuer
  Informatik, 2017.

\bibitem{kovacs}
Istv{\'a}n Kov{\'a}cs.
\newblock Indecomposable coverings with homothetic polygons.
\newblock {\em Discrete \& Computational Geometry}, 53(4):817--824, 2015.

\bibitem{unsplittable}
J{\'{a}}nos Pach and D{\"{o}}m{\"{o}}t{\"{o}}r P{\'{a}}lv{\"{o}}lgyi.
\newblock Unsplittable coverings in the plane.
\newblock In Ernst~W. Mayr, editor, {\em Graph-Theoretic Concepts in Computer
  Science - 41st International Workshop, {WG} 2015, Garching, Germany, June
  17-19, 2015, Revised Papers}, volume 9224 of {\em Lecture Notes in Computer
  Science}, pages 281--296. Springer, 2015.

\bibitem{surveycd}
J{\'a}nos Pach, D{\"o}m{\"o}t{\"o}r P{\'a}lv{\"o}lgyi, and G{\'e}za T{\'o}th.
\newblock Survey on decomposition of multiple coverings.
\newblock In {\em Geometry--Intuitive, Discrete, and Convex}, pages 219--257.
  Springer, 2013.

\bibitem{pachsharir2}
J{\'a}nos Pach and Micha Sharir.
\newblock On the boundary of the union of planar convex sets.
\newblock {\em Discrete \& Computational Geometry}, 21(3):321--328, 1999.

\bibitem{indec}
J{\'a}nos Pach, G{\'a}bor Tardos, and G{\'e}za T{\'o}th.
\newblock Indecomposable coverings.
\newblock In {\em Discrete Geometry, Combinatorics and Graph Theory}, pages
  135--148. Springer, 2007.

\bibitem{pinchasi}
Rom Pinchasi.
\newblock A finite family of pseudodiscs must include a
  {\textquotedblleft}small{\textquotedblright} pseudodisc.
\newblock {\em {SIAM} Journal on Discrete Mathematics}, 28(4):1930--1934, jan
  2014.

\bibitem{hananirado}
Fulek Radoslav, Michael~J Pelsmajer, Marcus Schaefer, and Daniel
  {\v{S}}tefankovi{\v{c}}.
\newblock Adjacent crossings do matter.
\newblock {\em Journal of Graph Algorithms and Applications}, 16(3):759--782,
  2012.

\bibitem{smorod-onthe}
Shakhar Smorodinsky.
\newblock On the chromatic number of geometric hypergraphs.
\newblock {\em SIAM Journal on Discrete Mathematics}, 21(3):676--687, 2007.

\bibitem{surveycf}
Shakhar Smorodinsky.
\newblock Conflict-free coloring and its applications.
\newblock In {\em Geometry--Intuitive, Discrete, and Convex}, pages 331--389.
  Springer, 2013.

\bibitem{snoeyink}
Jack Snoeyink and John Hershberger.
\newblock Sweeping arrangements of curves.
\newblock In {\em Proceedings of the fifth annual symposium on Computational
  geometry}, pages 354--363. ACM, 1989.

\bibitem{gezapc}
G\'eza T\'oth.
\newblock personal communication.

\bibitem{hanani2}
William~T Tutte.
\newblock Toward a theory of crossing numbers.
\newblock {\em Journal of Combinatorial Theory}, 8(1):45--53, 1970.

\end{thebibliography}
\end{document}